\documentclass[15pt]{article}

\usepackage{amsmath,amsfonts,amssymb,amsthm,amscd,mathrsfs}
\usepackage{enumitem}

\usepackage{yhmath}
\usepackage{graphicx} 

\newcounter{stepctr}
{\end{list}}

\newtheorem{thm}{Theorem}[section]
\newtheorem{prop}[thm]{Proposition}
\newtheorem{cor}[thm]{Corollary}

\theoremstyle{definition}

\newtheorem{ex}[thm]{Example}

\newtheorem{rema}[thm]{Remark}
\newtheorem{lem}[thm]{Lemma}
\newtheorem{prob*}{Open problem}
\newcommand{\demo}{\begin{proof}}

\newcommand{\A}{\mathcal{A}}

\newcommand{\R}{\ensuremath{\mathcal{R}}}

\newcommand{\N}{\mathbb{N}}

\newcommand{\C}{\mathbb{C}}

\usepackage{array,multirow,makecell}
\setcellgapes{1pt}
\makegapedcells
\newcolumntype{R}[1]{>{\raggedleft\arraybackslash }b{#1}}
\newcolumntype{L}[1]{>{\raggedright\arraybackslash }b{#1}}
\newcolumntype{C}[1]{>{\centering\arraybackslash }b{#1}}

\def\ll^2{{\mathcal L}(\ell^2(\N))}

\def\f^0x{{\mathcal F^0}(X) }

\title {\bf Generalization of some  commutative perturbation results }

\author{Zakariae Aznay, Abdelmalek Ouahab, Hassan Zariouh }

\date{}
\begin{document}
\maketitle \thispagestyle{empty}
\begin{abstract}\noindent\baselineskip=10pt
We study the stability of certain spectra under some algebraic conditions weaker than the commutativity and  we generalize many known    commutative perturbation results. 
\end{abstract}

 \baselineskip=15pt
  \baselineskip=15pt
\section{Introduction}
In  \cite{liu-wu-yu}, the authors  investigated how to explicitly express the Drazin inverse of the sum $(P+ Q)$ of two complex matrices $P$ and $Q,$ under the conditions $PQ \in \mbox{comm}(P)$  and $QP \in \mbox{comm}(Q),$  which are   weaker than the commutativity of $P$ with $Q.$  A few years later, Huihui Zhu et al.  \cite{zhu-chen-patricio}  obtained    the representations for the pseudo Drazin inverse of  the sum and the product of two elements of a  complex Banach algebra, under the same conditions. Recently,  H. Zou et al.    \cite{mosic-zou-chen}   extended the known expressions for the  generalized Drazin inverse of the product  and the sum  of two elements of a complex Banach algebra  by   considering the same conditions.\\
 In this paper, we    study  these conditions and  other  in a ring $\A,$ that are $ab \in \mbox{comm}(a),$   $ba \in \mbox{comm}(b),$ $ab \in \mbox{comm}(b)$  and $ba \in \mbox{comm}(a).$ After giving some algebraic results, we  focus on the Banach algebra of bounded linear operators $L(X)$ acting on the  complex Banach algebra $X.$ We  generalize some commutative  perturbation spectral results,  in particular, if   $N$ is nilpotent and $N\in\mbox{comm}_{r}(T)$ (i.e. $NT \in \mbox{comm}(T)$ and    $TN \in \mbox{comm}(N)$),  then      $\sigma_{*}(T)\setminus\{0\}=\sigma_{*}(T+N)\setminus\{0\},$ where  $\sigma_{*}\in\{\sigma_{p},\sigma_{p}^0,\sigma_{a}\}.$ If in addition     $N\in\mbox{comm}_{w}(T)$ (i.e. $N\in\mbox{comm}_{r}(T)$ and $N^{*} \in\mbox{comm}_{r}(T^{*})$), then   we deduce by duality  that   $\sigma_{*}(T)=\sigma_{*}(T+N),$  where $\sigma_{*}\in\{\sigma_{p},\sigma_{p}^{0},\sigma_{a},\sigma_{s},\sigma\}.$ This allows us to show that if $K$ is a power compact operator and 
$K\in \mbox{comm}_{w}(T)$, then       $\sigma_{*}(T)=\sigma_{*}(T+K),$ where  $\sigma_{*}\in\{\sigma_{uf},\sigma_{uw},\sigma_{ub}, \sigma_{e}, \sigma_{w},\sigma_{b} \}.$ Note that here we lose quite a few commutative properties, for example if $S\not\in\mbox{comm}(T)$ and $S\in\mbox{comm}_{w}(T),$ then  $S\not\in\mbox{comm}_{w}(T-\lambda I)$ for every  $\lambda \neq 0.$ 
\section{Terminology and preliminaries}
Let $\A$ be a ring and let $a\in \A.$ Denote by  $\mbox{comm}(a)$  the set of all elements that commute   with $a,$  by  $\mbox{comm}^{2}(a)=\mbox{comm}(\mbox{comm}(a))$ and by $\mbox{Nil}(\A)$ the nilradical of $\A.$ If in addition $\A$ is a complex Banach algebra with unit $e,$ then we means by $\sigma(a),$ $\mbox{acc}\,\sigma(a),$    $r(a)$ and $\mbox{exp}(a),$ the  spectrum, the accumulation point of $\sigma(a),$  the spectral radius of $a$   and the exponential of $a,$  respectively. We say that  $a$ is quasi-nilpotent  if $r(a)=0.$  In the case of $\A=L(X),$ the   algebra of all bounded linear operators acting on an infinite dimensional complex Banach space  $X,$    $T^{*},$ $\alpha(T)$ and  $\beta(T)$ means   respectively,     the dual of the operator $T\in \A,$   the dimension of the kernel $\mathcal{N}(T)$  and the codimension of the range $\R(T).$   Denote further  $\R(T^{\infty}):=\underset{n\geq0}{\bigcap}\R(T^{n})$ and $\mathcal{N}(T^{\infty}):=\underset{n\geq0}{\bigcup}\mathcal{N}(T^{n}).$ The ascent and the descent of $T$ are defined  by $p(T)=\inf\{n\in \mathbb{N}: \mathcal{N}(T^n) = \mathcal{N}(T^{n+1})\}$ (with $\mbox{inf}\emptyset=\infty$)    and  $q(T)= \inf\{n\in \mathbb{N}: \R(T^n) = \R(T^{n+1})\}.$  A  subspace $M$ of $X$ is $T$-invariant if $T(M)\subset M$ and   the restriction of $T$ on $M$ is denoted by $T_{M}.$   $(M,N) \in \mbox{Red}(T)$  if $M,$   $N$ are  closed $T$-invariant subspaces and  $X=M\oplus N$      ($M\oplus N$ means that $M\cap N=\{0\}$).    Let  $n\in\N,$   denote by   $T_{[n]}=T_{\mathcal{R}(T^{n})}$    and by     $m_T=\mbox{inf}\{n \in \N  :  \mbox{inf}\{\alpha(T_{[n]}),\beta(T_{[n]})\}<\infty\}$   the \textit{essential  degree} of   $T.$   $T$ is called upper semi-B-Fredholm  (resp., lower semi-B-Fredholm)  if the \textit{essential ascent} $p_{e}(T):=\mbox{inf}\{n \in \N  :  \alpha(T_{[n]})<\infty\}<\infty$ and $\R(T^{p_{e}(T)+1})$ is closed  (resp.,  the \textit{essential descent}  $q_{e}(T):=\mbox{inf}\{n \in \N  :  \beta(T_{[n]})<\infty\}<\infty$ and $\R(T^{q_{e}(T)})$ is closed). If  $T$ is   an upper or a lower  (resp.,  upper and   lower)  semi-B-Fredholm then $T$ it is called   \emph{semi-B-Fredholm} (resp., \emph{B-Fredholm})  and      its     index   is  defined   by $\mbox{ind}(T) = \alpha(T_{[m_{T}]})-\beta(T_{[m_{T}]}).$ $T$ is said to be an upper semi-B-Weyl (resp., a lower semi-B-Weyl, B-Weyl,  left Drazin invertible, right Drazin invertible, Drazin invertible) if  $T$ is an upper semi-B-Fredholm with $\mbox{ind}(T)\leq 0$  (resp., $T$ is a lower  semi-B-Fredholm with $\mbox{ind}(T)\geq 0,$ $T$ is a B-Fredholm with $\mbox{ind}(T)=0,$   $T$ is an upper semi-B-Fredholm and $p(T_{[m_{T}]})<\infty,$ $T$ is a lower  semi-B-Fredholm and $q(T_{[m_{T}]})<\infty,$  $p(T_{[m_{T}]})=q(T_{[m_{T}]})<\infty$).   If $T$ is  upper semi-B-Fredholm (resp.,  lower semi-B-Fredholm,    semi-B-Fredholm, B-Fredholm, upper semi-B-Weyl,  lower semi-B-Weyl, B-Weyl,  left Drazin invertible, right Drazin invertible, Drazin invertible) with essential degree $m_{T}=0,$ then $T$ is said to be  an upper semi-Fredholm (resp.,  lower semi-Fredholm,   semi-Fredholm, Fredholm, upper semi-Weyl, lower semi-Weyl, Weyl, upper semi-Browder, lower semi-Browder, Browder) operator.  $T$ is said to be  bounded below if $T$ is upper semi-Fredholm with $\alpha(T)=0.$ \\
$\sigma_{p}(T)$:  point  spectrum of $T$\\
$\sigma_{p}^{0}(T):=\{\lambda \in \sigma_{p}(T): \alpha(T-\lambda I )<\infty\}$\\
$\sigma_{a}(T)$:  approximatif  spectrum of $T$\\
   $\sigma_{s}(T)$:  surjectif spectrum of $T$\\
$\sigma_{e}(T)$:  essential  spectrum of $T$\\
   $\sigma_{uf}(T)$:  upper semi-Fredholm spectrum of $T$\\
   $\sigma_{lf}(T)$:  lower semi-Fredholm spectrum of $T$\\
   $\sigma_{w}(T)$: Weyl spectrum of $T$\\
   $\sigma_{uw}(T)$: upper semi-Weyl spectrum of $T$\\
   $\sigma_{lw}(T)$: lower semi-Weyl spectrum of $T$\\
   $\sigma_{b}(T)$:  Browder spectrum of $T$\\
   $\sigma_{ub}(T)$:  upper semi-Browder spectrum of $T$\\
   $\sigma_{lb}(T)$:  lower semi-Browder spectrum of $T$\\
   $\sigma_{gd}(T)=\mbox{acc}\,\sigma(T)$ is  the   generalized Drazin spectrum of  $T$\\
$\sigma_{g_{z}d}(T)=\mbox{acc}\,\mbox{acc}\,\sigma(T)$  is the   $g_{z}$-invertible   spectrum of $T$ \cite{aznay-ouahab-zariouh6}

\section{Generalization of Newton formula}
We start  this section by the next preliminary lemma that  will be play a crucial role in the sequel.
\begin{lem}\label{lem1} Let  $\A$ be a ring and  $a,b \in \A.$  The following statements  hold:
\begin{enumerate}[nolistsep]
\item[(I)] If $ab \in \mbox{comm}(a),$ then
\begin{enumerate}[nolistsep]
\item[(i)]   $a^{n}b \in \mbox{comm}(a^{m})$  for every  integers $n,m\geq1.$
\item[(ii)]    $(ab)^{n}=a^{n}b^{n}$ and $(ba)^{n}=ba^{n}b^{n-1}=(ba)(ab)^{n-1}=ba^{n-1}b^{n-1}a$ for every  integer $n\geq 2.$
\item[(iii)]   $a(a+b) \in \mbox{comm}(a).$
\end{enumerate}
\item[(II)]  If $ab \in \mbox{comm}(b),$ then
\begin{enumerate}[nolistsep]
\item[(i)]    $ab^{n} \in \mbox{comm}(b^{m})$  for every  integers $n,m\geq1.$
\item[(ii)]   $(ab)^{n}=a^{n}b^{n}$ and $(ba)^{n}=a^{n-1}b^{n}a=(ab)^{n-1}(ba)=ba^{n-1}b^{n-1}a$  for every  integer $n\geq 2.$
\item[(iii)]  $(a+b)b \in \mbox{comm}(b).$
\end{enumerate}
\item[(III)]    If $ab \in \mbox{comm}(a)$  and $ba \in \mbox{comm}(b),$ then
\begin{enumerate}[nolistsep]
\item[(i)]   $a^{n}b^{m} \in \mbox{comm}(a^{k})$   for every strictly positive   integers $n,m$ and $k.$
\item[(ii)]    $a^{n}- b^{n}+ ba^{n-1}- a^{n-1}b= (a^{n-1}+ ba^{n-2}+ b^2a^{n-3}+ \dots+ b^{n-2}a+ b^{n-1})(a- b)$ and $a^{n}- b^{n}+ b^{n-1}a- ab^{n-1}= (a^{n-1}+ a^{n-2}b+ a^{n-3}b^{2}+ \dots+ ab^{n-2}+ b^{n-1})(a- b)$   for every integer $n\geq 2.$
\item[(iii)]   $(a+b)a \in \mbox{comm}(a+b)$ and $b(a+b) \in \mbox{comm}(b).$
\end{enumerate}
\item[(IV)]    If $ab \in \mbox{comm}(b)$  and $ba \in \mbox{comm}(a),$ then
\begin{enumerate}[nolistsep]
\item[(i)] $a^{n}b^{m} \in \mbox{comm}(b^{k})$    for every strictly positive   integers $n,m$ and $k.$
\item[(ii)]     $a^{n}- b^{n}+ ab^{n-1}- b^{n-1}a= (a- b)(a^{n-1}+ ba^{n-2}+ b^2a^{n-3}+ \dots+ b^{n-2}a+ b^{n-1})$ and  $a^{n}- b^{n}+ a^{n-1}b- ba^{n-1}= (a- b)(a^{n-1}+ a^{n-2}b+ a^{n-3}b^{2}+ \dots+ ab^{n-2}+ b^{n-1})$    for every integer $n\geq 2.$
\item[(iii)] $a(a+b) \in \mbox{comm}(a+b)$ and  $(a+b)b \in \mbox{comm}(b).$
\end{enumerate}
\end{enumerate}
\end{lem}

\begin{proof}
(I)  (i)  Let's use induction with the following statement  \[P_{m} :  \quad  a^{n}b  \in \mbox{comm}(a^{m}) \text{ for every integer } n\geq1.\]
   Let $n\geq 1$ be an integer  such that  $a^{n}ba=a^{n+1}b.$ Then $ a^{n+1}ba=a(a^{n}ba)=a(a^{n+1}b)=a^{n+2}b.$   So  $P_{1}$ holds.   Assume that $P_{m}$ holds for some integer $m\geq 1.$ Let $n\geq1$ be an integer, then  $(a^{n}b)a^{m+1}=((a^{n}b)a^{m})a=(a^{m}(a^{n}b))a=a^{m}((a^{n}b)a)=a^{m}(a(a^{n}b))=a^{m+1}(a^{n}b).$ So $P_{m+1}$ holds.
Consequently,   $a^{n}b \in \mbox{comm}(a^{m})$ for   every  integers $n,m\geq1.$\\
(ii) The equality $(ab)^{n}=a^{n}b^{n}$ is obvious. Let us  prove  by induction that  $(ba)^{n}=ba^{n}b^{n-1}$ for every integer $n\geq2.$ For $n=2$ the equality  holds.  Assume   that $(ba)^{n}=ba^{n}b^{n-1}$ for some integer $n\geq2,$ then we get from the first  point  (i) that  $(ba)^{n+1}=(ba)(ba)^{n}=b(aba^{n})b^{n-1}=b(a^{n+1}b)b^{n-1}=ba^{n+1}b^{n}.$   Consequently, the equality holds for every integer $n\geq2.$ On the other hand, since       $(ba)^{n}=b(ab)^{n-1}a$ is always true,   then $(ba)^{n}=ba^{n}b^{n-1}=(ba)(ab)^{n-1}=b(ab)^{n-1}a=ba^{n-1}b^{n-1}a$ for every  integer $n\geq 2.$ The point (iii) is trivial.\\
(II) The proof is identical  to that of (I).\\
(III) (i)  Let's  use induction  with the following  statement   \[P_{k} : a^{n}b^{m} \in \mbox{comm}(a^{k})  \text{ for every  strictly positive  integers } n \text{ and }m.\]
  $P_{1}$ is holds. Indeed,  we  consider the following statement  \[Q_{m} : a^{n}b^{m} \in \mbox{comm}(a)  \text{ for every strictly positive  integer } n.\]  From  the point   (i) of  the statement (I),  $Q_{1}$ holds. Assume now  that  $Q_{m}$ holds for some strictly integer $m$  and let $n$ be a strictly positive integer. Since $bab=b^{2}a,$ we conclude that   $a(a^{n}b^{m+1})=(a(a^{n}b^{m}))b=((a^{n}b^{m})a)b=a^{n}(b^{m}ab)=a^{n}(b^{m+1}a)=(a^{n}b^{m+1})a.$  So  $Q_{m+1}$ holds.\\
 Suppose  that $P_{k}$ holds for some strictly positive  integer $k$ and     let $n$ and   $m$ be two strictly positive integers.   The  statement  $P_{1}$ implies that  $(a^{n}b^{m})a^{k+1}=((a^{n}b^{m})a^{k})a=(a^{k}(a^{n}b^{m}))a=a^{k}((a^{n}b^{m})a)=a^{k}(a(a^{n}b^{m}))=a^{k+1}(a^{n}b^{m}).$ Thus  $P_{k+1}$ holds and this  completes the proof.\\
The point (ii)  is an immediate  consequence of (i), and the point (iii) is trivial.\\
(IV) Goes similarly with (III).
\end{proof}
\medskip
 Throughout this paper we consider  on a ring  $\A$ the sets     defined as follows
    $$\mbox{comm}_{l}(a)=\{b \in \A :  ab \in \mbox{comm}(a) \text{ and } ba  \in \mbox{comm}(b)\},$$
  $$\mbox{comm}_{r}(a) =\{b \in \A :    ab \in \mbox{comm}(b) \text{ and  } ba \in \mbox{comm}(a)\},$$
  $$\mbox{comm}_{w}(a)=\mbox{comm}_{l}(a)\cap\mbox{comm}_{r}(a).$$

\begin{ex}\label{s.ex}  Let $\A$ be a ring. Then for every $a,b \in \A$ we have
$$a \in \mbox{comm}(b) \Longrightarrow  b \in \mbox{comm}_{w}(a) \Longrightarrow b^{2}  \in \mbox{comm}(a) \text{ and } a^{2}  \in \mbox{comm}(b)$$
However, we show by  the following examples  that the reverse  of these implications are  not true in general.
\begin{enumerate}[nolistsep]
\item[(i)]   In  the matrix space $\mathcal{M}_{2}(\A),$  where $\A$ is a ring with non-null unit $e,$  the elements   $P=\left(\begin{matrix} 0 & e \\ 0 & 0  \end{matrix}\right)$ and $Q= \left(\begin{matrix} 0 & e \\ 0 & e  \end{matrix}\right)$ satisfy   $PQP=P^{2}Q=QP^{2}=QPQ=Q^{2}P\neq PQ^{2},$ but $PQ\neq QP.$ On the other hand, if we consider $S=\left(\begin{matrix} 0 & 0 \\ e & 0  \end{matrix}\right),$ then $S \in \mbox{comm}(P^{2})$ and $P \in \mbox{comm}(S^{2}),$ but $PS$ does not commute neither with $P$ nor with $S,$ and   $SP$ does not commute neither with $P$ nor with $S.$ 
\item[(ii)]   Hereafter $\ell^{2}$ denotes the Hilbert space $\ell^{2}(\N).$ We consider  in  the Banach algebra $L(\ell^{2}),$  the operators $T$ and $S$ defined by $T(x_1, x_2, \dots)=(x_1, x_1, x_3, x_4,  \dots)$ and $S(x_1, x_2, \dots)=(x_1, 0, \dots).$  Then  $S \in \mbox{comm}_{l}(T),$  but  $ST$ does not commute with  $T$ and  $TS$ does not commute with  $S.$ On the other hand, $S^{*} \in \mbox{comm}_{r}(T^{*}),$  but   $T^{*}S^{*}$ does not commute with  $T^{*}$ and  $S^{*}T^{*}$ does not commute with  $S^{*}.$ This shows  that the  conditions assumed in the statements (I) and (II) of Lemma \ref{lem1} are independent.     As another example, consider in    $\mathcal{M}_2(A),$    $M=\left(\begin{matrix} e & e \\ 0 & 0  \end{matrix}\right)$ and $N= \left(\begin{matrix} a & a \\ b & b  \end{matrix}\right).$ Then    $M\in\mbox{comm}_{l}(N)$   for every $a,b \in \A$ such that $b\neq 0$ and $a\neq -b.$ But $NM$ does not commute with  $M$ and  $MN$ does not commute with  $N.$
\item[(iii)] Let $T$ and $N$ be the operators defined on $\ell^{2}$  by $T(x_{1}, x_{2}, \dots)=(x_{2}, 0, \dots),$  $N(x_{1}, x_{2}, \dots)=(0, x_{1}, 0,\dots).$ Then $T \in \mbox{comm}(S^{2})$ and $S \in \mbox{comm}(T^{2}).$ But  $TN\notin \mbox{comm}(T)\cup \mbox{comm}(N)$ and $NT\notin \mbox{comm}(T)\cup \mbox{comm}(N).$
\item[(iv)] For the operators $N_{1}$ and $N_{2}$ defined  on $\ell^{2}$ by $N_{1}(x_1, x_2, \dots)=(0, x_{1}, x_{2}, 0, \dots),$  $N_{2}(x_1, x_2, \dots)=(0, -x_{1}, 0,  \dots),$  we have   $N_{1} \in \mbox{comm}_{w}(N_{2}),$  but   $N_{1} \notin \mbox{comm}(N_{2}).$
\item[(v)]    In  $\mathcal{M}_3(\C),$      $P=\left(\begin{matrix} 0 & 0 & 0 \\ \frac{1}{2} & 0 & 0 \\ 0 & \frac{1}{3} & 0  \end{matrix}\right)$ and $Q=\left(\begin{matrix} 0 & 0 & 0 \\ -\frac{1}{2} & 0 & 0 \\ 0 & 0 & 0  \end{matrix}\right)$ satisfy   $PQP=P^{2}Q=QP^{2}=QPQ=Q^{2}P=PQ^{2}=0,$ so that $P \in\mbox{comm}_{w}(Q),$  but $PQ\neq QP.$ 
\end{enumerate}
\end{ex}
A simple check, one can easily obtain the dressed results in the following remark.
\begin{rema}\label{rema1} Let  $\A$ be a ring with unit $e.$ For      $a,b \in \A$ and $\mu,\lambda \in \C,$  the following statements  hold:\\
(i)  If $ab \in \mbox{comm}(a),$ then     $(a-\lambda e)(b-\mu e) \in \mbox{comm}(a-\lambda e)$ if and only if $a \in \mbox{comm}(b)$ or $\lambda =0.$ \\
(ii)  If $ba \in \mbox{comm}(a),$ then   $(b-\mu e)(a-\lambda e) \in \mbox{comm}(a-\lambda e)$ if and only if $a \in \mbox{comm}(b)$ or $\lambda =0.$\\
(iii)   If $b \in \mbox{comm}_{w}(a),$ then $a^{n} \in \mbox{comm}(b^{m})$ for every integers $n,m\geq 1$ such that  $nm\geq 2.$\\
(iv)  If $b \in \mbox{comm}_{l}(a)$ (resp., $b \in \mbox{comm}_{r}(a),$ $b \in \mbox{comm}_{w}(a)$), then $(a+b) \in \mbox{comm}_{l}(b)$   (resp., $(a+b) \in \mbox{comm}_{r}(b),$ $(a+b) \in \mbox{comm}_{w}(b)$).\\
(v)   If $aba=a^{2}b=ba^{2},$ then $a^{n} \in \mbox{comm}(b^{m})$ for every  $m\geq 1$ and  $n\geq 2.$
\end{rema}
Let $\A$ be a  unital complex Banach algebra and $B\subset \A.$  The exponential of $a \in \A$ is defined by   $
\mbox{exp}(a)=\displaystyle\sum_{n=0}^{\infty}\frac{a^{n}}{n!}.$  In the next  we denote by  $$C_{1}(B)=\{a\in \A: \forall b\in B,\,   a \in \mbox{comm}(ab)\cup\mbox{comm}(ba) \text{ or } b\in \mbox{comm}(ba)\},$$  $$C_{2}(B)=\{a\in \A: \forall b\in B,\,   a \in \mbox{comm}(ab)\cup\mbox{comm}(ba) \text{ or } b\in \mbox{comm}(ab)\},$$  $$C_{3}(B)=\{a\in \A: \forall b\in B,\,   b \in \mbox{comm}(ab)\cup\mbox{comm}(ba)\}.$$
\begin{thm} If  $\A$ is a unital complex Banach algebra  and  $\A=C_{i}(\A),$ i=1, 2 or 3, then  $\A$ is commutative.
\end{thm}
\begin{proof}  Let $a,b\in \A$ and let $\lambda \in \C.$ Consider $x_{\lambda}=\mbox{exp}(\lambda a)$ and $y_{\lambda}=\mbox{exp}(-\lambda a)b.$ Assume that  $\A=C_{1}(\A),$  then    $b\mbox{exp}(\lambda a)=x_{\lambda}y_{\lambda}x_{\lambda}\in\{x_{\lambda}^{2}y_{\lambda},y_{\lambda}x_{\lambda}^{2}\}=\{\mbox{exp}(\lambda a)b,\mbox{exp}(-\lambda a)b\mbox{exp}(2\lambda a)\}$ or $\mbox{exp}(-\lambda a)b^{2}=y_{\lambda}x_{\lambda}y_{\lambda}=y_{\lambda}^{2}x_{\lambda}=\mbox{exp}(-\lambda a)b\mbox{exp}(-\lambda a)b\mbox{exp}(\lambda a).$  Hence $b^{2}=b\mbox{exp}(-\lambda a)b\mbox{exp}(\lambda a)$ for all $\lambda \in \C.$  Moreover, $$b\mbox{exp}(\lambda a)b\mbox{exp}(-\lambda a)=b\displaystyle\sum_{n=0}^{\infty}\frac{\lambda^{n}}{n!}\left(\displaystyle\sum_{k=0}^{n}C_{n}^{k}a^{k}b(-a)^{n-k}\right)=\displaystyle\sum_{n=0}^{\infty}\frac{\lambda^{n}}{n!}b(\delta_{a})^{n}(b),$$
where ${\displaystyle C_{n}^{k}={\frac {n!}{k!\,(n-k)!}}}$   and  $\delta_{a}(b)=ab-ba.$  Thus $bab=b^{2}a$ for every $a,b\in \A.$  By the similar  arguments, we get      $b=\mbox{exp}(\lambda a)b\mbox{exp}(-\lambda a)=\displaystyle\sum_{n=0}^{\infty}\frac{\lambda^{n}}{n!}\left(\displaystyle\sum_{k=0}^{n}C_{n}^{k}a^{k}b(-a)^{n-k}\right)=\displaystyle\sum_{n=0}^{\infty}\frac{\lambda^{n}}{n!}(\delta_{a})^{n}(b),$ for all $\lambda \in \C.$ Therefore $ab=ba$ and  $\A$ is commutative. The  case of $\A=C_{2}(\A)$ is  analogous. For the case of   $\A=C_{3}(\A),$ it suffices to use that same arguments with    $x_{\lambda}=\mbox{exp}(\lambda a)$ and $y_{\lambda}=b\mbox{exp}(-\lambda a).$
\end{proof}
\begin{thm}\label{thmp} Let $\A$ be a ring. For every  $a,b\in \A$ we have\\
(i) If     $b \in \mbox{comm}_{r}(a),$ then  for every integer $n>0,$  $\displaystyle(a+b)^{n}=\sum_{k=1}^{n} C_{n-1}^{k-1}\left(a^{n-k}b^{k}+b^{n-k}a^{k}\right).$\\
(ii)  If $b \in \mbox{comm}_{l}(a),$ then for every integer $n>0,$  $(a+b)^{n}=\displaystyle\sum_{k=1}^{n} C_{n-1}^{k-1}\left(a^{k}b^{n-k}+b^{k}a^{n-k}\right).$\\
Where   $a^{0}$ designates   the unit element of $\A.$
\end{thm}
\begin{proof}
(i)  For $n=1$ the statement holds. Assume that the statement holds for some   integer  $n\geq 1,$ then
\begin{align*}
 (a+b)^{n+1}&=(a+b)(a+b)^{n}=(a+b)\left(\sum_{k=1}^{n} C_{n-1}^{k-1}\left(a^{n-k}b^{k}+b^{n-k}a^{k}\right)\right)\\
&=a\left(\sum_{k=1}^{n} C_{n-1}^{k-1}\left(a^{n-k}b^{k}+b^{n-k}a^{k}\right)\right)+b\left(\sum_{k=1}^{n} C_{n-1}^{k-1}\left(a^{n-k}b^{k}+b^{n-k}a^{k}\right)\right)\\
&=\sum_{k=1}^{n} C_{n-1}^{k-1}\left(a^{n+1-k}b^{k}+ab^{n-k}a^{k}\right)+\sum_{k=1}^{n} C_{n-1}^{k-1}\left(ba^{n-k}b^{k}+b^{n+1-k}a^{k}\right)\\
&=\sum_{k=1}^{n} C_{n-1}^{k-1}\left(a^{n+1-k}b^{k}+b^{n+1-k}a^{k}\right)+\sum_{k=1}^{n} C_{n-1}^{k-1}\left(ab^{n-k}a^{k}+ba^{n-k}b^{k}\right)\\
&=\sum_{k=1}^{n} C_{n-1}^{k-1}\left(a^{n+1-k}b^{k}+b^{n+1-k}a^{k}\right)+\sum_{k=1}^{n} C_{n-1}^{k-1}\left(b^{n-k}a^{k+1}+a^{n-k}b^{k+1}\right)  \quad (\text{see Lemma \ref{lem1}})\\
&=\sum_{k=1}^{n} C_{n-1}^{k-1}\left(a^{n+1-k}b^{k}+b^{n+1-k}a^{k}\right)+\sum_{k=2}^{n+1} C_{n-1}^{k-2}\left(b^{n+1-k}a^{k}+a^{n+1-k}b^{k}\right)\\
&=\sum_{k=2}^{n} \left( C_{n-1}^{k-2}+ C_{n-1}^{k-1}  \right)\left(a^{n+1-k}b^{k}+b^{n+1-k}a^{k}\right)+\left(a^{n}b+b^{n}a\right) +\left(a^{n+1}+b^{n+1}\right)\\
&=\sum_{k=1}^{n+1} C_{n}^{ k-1}\left(a^{n+1-k}b^{k}+b^{n+1-k}a^{k}\right).
\end{align*}
So the statement holds for $n+1.$ Consequently, $\displaystyle(a+b)^{n}=\sum_{k=1}^{n} C_{n-1}^{k-1}\left(a^{n-k}b^{k}+b^{n-k}a^{k}\right)$ for every  integer $n>0.$\\
(ii) The second statement can be proved similarly.
\end{proof}
\par Note that the second point of the previous theorem  was firstly  proved for complex matrices in \cite[Lemma 2.3]{liu-wu-yu}. This result has been extended  by  Huihui Zhu et al. and  Honglin Zou et al. to a Banach algebra, see \cite[Lemma 2.3]{zhu-chen-patricio} and \cite[Lemma 2.9]{mosic-zou-chen}.
\medskip
\par The next  corollary gives a generalization to   the   known Binomial theorem.
\begin{cor}\label{cornewton}  Let $\A$ be a ring. If  $a,b\in \A$ such that $b \in \mbox{comm}_{w}(a),$ then for every positive  integer   $n\neq  2,$ the  following statements hold:\\
(i)  $(a+b)^{n}=\displaystyle\sum_{k=0}^{n} C_{n}^{k} a^{k}b^{n-k}=\sum_{k=0}^{n} C_{n}^{k} b^{k}a^{n-k}.$\\
(ii) $a^{n}- b^{n}=(a-b)\displaystyle\sum_{k=0}^{n-1}a^{k}b^{n-k-1}=\left(\sum_{k=0}^{n-1}a^{k}b^{n-k-1}\right)(a-b)=(a-b)\sum_{k=0}^{n-1}b^{n-k-1}a^{k}=\left(\sum_{k=0}^{n-1}b^{n-k-1}a^{k}\right) (a-b).$
\end{cor}
\begin{proof} (i) The cases $n=0$ and $n=1$ are trivial. For  $n\geq 3,$  as  $b \in \mbox{comm}_{w}(a)$ then  Theorem  \ref{thmp} and Remark \ref{rema1}   imply  that
\begin{align*}
(a+b)^{n}&=\sum_{k=1}^{n} C_{n-1}^{k-1}\left(a^{n-k}b^{k}+b^{n-k}a^{k}\right)\\
&=\sum_{k=1}^{n} C_{n-1}^{k-1}a^{n-k}b^{k}+\sum_{k=1}^{n} C_{n-1}^{k-1}b^{n-k}a^{k}\\
&=\sum_{k=1}^{n} C_{n-1}^{k-1}a^{n-k}b^{k}+\sum_{k=0}^{n-1} C_{n-1}^{n-k-1}b^{k}a^{n-k}\\
&=\sum_{k=1}^{n-1} (C_{n-1}^{k-1}+C_{n-1}^{n-k-1})a^{n-k}b^{k}+a^{n}+b^{n}\\
&=\sum_{k=0}^{n} C_{n}^{k} a^{k}b^{n-k}=\sum_{k=0}^{n} C_{n}^{k} b^{k}a^{n-k}
\end{align*}
(ii)  Follows directly from  Lemma \ref{lem1} and  Remark \ref{rema1}.
\end{proof}
Let $\A$ be a Banach algebra with unit $e.$   It is well known that      $\mbox{exp}(a+b)=\mbox{exp}(a)\mbox{exp}(b)$ for every $a\in \mbox{comm}(b).$ But this identity  can fail for noncommuting  $a$ and $b.$ The next corollary shows that  if  $b \in \mbox{comm}_{w}(a),$ then this identity remains true if and only if $a \in \mbox{comm}(b).$
\begin{cor}   Let $\A$ be a Banach algebra with unit $e$  and let  $a,b\in \A$ such that $b \in \mbox{comm}_{w}(a).$ Then $\displaystyle \mbox{exp}(a)\mbox{exp}(b)-\mbox{exp}(a+b)=\frac{ab-ba}{2}.$ In particular, $\displaystyle \mbox{exp}(a)\mbox{exp}(b)-\mbox{exp}(b)\mbox{exp}(a)=ab-ba.$
\end{cor}
\begin{proof} The   Cauchy product   implies  that
\begin{align*}
\mbox{exp}(a)\mbox{exp}(b)&=\left(\sum_{n=0}^{\infty}\frac{a^{n}}{n!}\right)\left(\sum_{n=0}^{\infty}\frac{b^{n}}{n!}\right)\\
&=\sum_{n=0}^{\infty}\left(\sum_{k=0}^n\frac{a^{k}}{k!}\frac{b^{n-k}}{(n-k)!}\right)\\
&=e+(a+b)+\frac{a^{2}+2ab+b^{2}}{2}+\sum_{n=3}^{\infty}\left(\sum_{k=0}^n\frac{a^{k}}{k!}\frac{b^{n-k}}{(n-k)!}\right)\\
&=e+(a+b)+\frac{a^{2}+2ab+b^{2}}{2}+\sum_{n=3}^{\infty}\frac{(a+b)^{n}}{n!}\\
&=\sum_{n=0}^{\infty}\frac{(a+b)^{n}}{n!}+\frac{ab-ba}{2}\\
&=\mbox{exp}(a+b)+\frac{ab-ba}{2}.
\qedhere
\end{align*}
\end{proof}
Recall that an element   $x$  of a ring   $\A$ is called  nilpotent   if    $x^{n}=0$  for some positive integer $n.$  If so then the integer $d(x)=\mbox{min}\{n\in \mathbb{N}: x^{n}=0\}$ is called the degree of $x.$ And   the nilradical $\mbox{Nil}(\A)$ $\A$ is the set consisting of all  nilpotent elements of $\A$, that is, $\mbox{Nil}(\A):=\{a\in \A \text{ }|\text{ } a \text{ is nilpotent}\}.$
\begin{lem}\label{lemnilp}  Let $\A$ be a ring   and  let     $a,b \in \A.$ The following assertions hold:\\
(i) If    $b\in\mbox{Nil}(\A)$   and    $ab \in \mbox{comm}(a)\cup \mbox{comm}(b)$  or   $ba \in \mbox{comm}(a)\cup \mbox{comm}(b),$  then   $ab$ and $ba$ both belong to $\mbox{Nil}(\A).$ Furthermore, in the first case  we have  $d(ab)\leq d(b)$ and $d(ba)\leq d(b)+1,$ and in the second case we have $d(ba)\leq d(b)$ and $d(ab)\leq d(b)+1.$\\
(ii) If  $a,b\in\mbox{Nil}(\A)$  and   $b \in \mbox{comm}_{l}(a)\cup\mbox{comm}_{r}(a),$  then $a+b\in\mbox{Nil}(\A)$ and   $$\mbox{max}\{d(a),d(b)\}-\mbox{min}\{d(a),d(b)\}\leq d(a+b)\leq d(a)+d(b).$$
(iii)   If $b \in \mbox{comm}_{l}(a)\cap\mbox{comm}(a^{n})$ for some integer $n>0,$ then     $a^{m}- b^{m}= (a^{m-1}+ ba^{m-2}+ b^2a^{m-3}+ \dots+ b^{m-2}a+ b^{m-1})(a- b)$ for all   integer  $m>n.$ Analogously, if $a \in \mbox{comm}_{r}(b)\cap \mbox{comm}(b^{n})$ for some integer $n>0,$ then     $a^{m}- b^{m}= (a- b)(a^{m-1}+ ba^{m-2}+ b^2a^{m-3}+ \dots+ b^{m-2}a+ b^{m-1})$ for all  integer  $m>n.$
\end{lem}
\begin{proof} (i) Let $n>0$ be an integer such that $b^{n}=0.$ If  $ab \in \mbox{comm}(a),$ then  Lemma \ref{lem1} implies that  $(ab)^{n}=a^{n}b^{n}=0$ and $(ba)^{n+1}=ba^{n+1}b^{n}=0.$ And if  $ba \in \mbox{comm}(a),$       we obtain again by  Lemma  \ref{lem1}  that $(ba)^{n}=b^{n}a^{n}=0$ and $(ab)^{n+1}=b^{n}a^{n}b=0.$ The other cases go similarly.\\
(ii) Let $n,m\geq 1$   two integers such that $a^{n}=0$ and  $b^{m}=0.$ If  $b \in \mbox{comm}_{l}(a),$ from Theorem \ref{thmp} we get   $(a+b)^{n+m}=\displaystyle\sum_{k=1}^{n+m} C_{n+m-1}^{k-1}\left(a^{k}b^{n+m-k}+b^{k}a^{n+m-k}\right).$ Thus  if $k\geq n,$ then   $a^{k}=0$ and if $k\leq n,$ then $n+m-k\geq m$ and so $b^{n+m-k}=0.$   If $k\geq m,$ then    $b^{k}=0$ and if $k\leq m,$ then $n+m-k\geq n,$  so $a^{n+m-k}=0.$  Hence $(a+b)^{n+m}=0$ and consequently  $d(a+b)\leq d(a)+d(b).$  On the other hand,  we have from Remark \ref{rema1}  that   $(a+b)\in\mbox{comm}_{l}(b).$ Hence  $\mbox{max}\{d(a),d(b)\}-\mbox{min}\{d(a),d(b)\}\leq d(a+b).$ The proof of  the case  $b \in \mbox{comm}_{r}(a)$ goes similarly. \\
(iii) Is an immediate consequence of Lemma \ref{lem1}.
\end{proof}



\medskip
Let $a$ be an element of a Banach algebra $\A$ with unit $e.$  The spectral radius $r(a)$  of $a$  can be expressed by the formula   $r(a)=\mbox{inf}\{M>0 : \left((\frac{a}{M})^{n}\right)_{n} \text{ is bounded}\}.$ 
\begin{prop}\label{propradspec} Let  $\A$ be  a complex Banach algebra  with unit $e$ and let $a,b\in \A.$   The following assertions hold:\\
(i) If $ab \in \mbox{comm}(a)\cup\mbox{comm}(b),$ then $r(ab)\leq r(a)r(b).$\\
(ii) If $a\in\mbox{comm}_{r}(b)\cup\mbox{comm}_{l}(b),$ then $r(a+b) \leq r(a)+r(b).$
\end{prop}
\begin{proof}  Let   $M > r(a),$  $N > r(b).$
(i)  From       Lemma \ref{lem1} we have  $(ab)^{n}=a^{n}b^{n}.$ Since the product of two bounded sequences is bounded, then  the sequence $\{(\frac{a^{p}b^{q}}{M^{p}N^{q}})\}_{p,q} $  is  bounded.   In particular,  $\{(\frac{ab}{MN})^{n}\}_{n}$ is  bounded  and hence $r(ab) \leq r(a)r(b).$ \\
(ii) Assume that $a\in\mbox{comm}_{r}(b)$ (the other case goes similarly). From  Theorem \ref{thmp}, we have
$$\left(\frac{a+b}{M+N}\right)^{n}=\displaystyle\sum_{k=1}^{n} C_{n-1}^{k-1}\left(\frac{M^{n-k}N^{k}}{(M+N)^{n}}\frac{a^{n-k}b^{k}}{M^{n-k}N^{k}}+\frac{N^{n-k}M^{k}}{(M+N)^{n}}\frac{b^{n-k}a^{k}}{N^{n-k}M^{k}}\right).$$ Hence $\{ (\frac{a+b}{M+N})^{n}\}_{n}$
is bounded and thus  $r(a+b) \leq r(a)+r(b).$  
\end{proof}
\begin{cor}\label{corquasi}
Let  $\A$ be  a complex Banach algebra  with unit $e$ and let $a,b\in \A.$  The following assertions hold:\\
(i) If   $a$ or $b$ is quasi-nilpotent and  $ab \in \mbox{comm}(a)\cup\mbox{comm}(b),$   then  $ab$  is  quasi-nilpotent.\\
(ii) If  $a$ and  $b$ are quasi-nilpotent and   $a\in\mbox{comm}_{r}(b)\cup\mbox{comm}_{l}(b),$ then   $a+b$ is quasi-nilpotent.
\end{cor}

\section{   Perturbation results}
Throughout  this   section,   we focus on the stability of  some spectra of   bounded linear operators in the  Banach algebra $\A=L(X).$  We start first with some  preliminaries results.
\begin{prop}\label{propn.1}
Let $ S, T \in L(X).$ The following statements hold:\\
(i) $TS \in \mbox{comm}(T)$ if and only if $S^{*}T^{*} \in \mbox{comm}(T^{*}).$\\
(ii)  $TS \in \mbox{comm}(T)$  if and only if  $\R(ST-TS)\subset \mathcal{N}(T),$ and $ST \in \mbox{comm}(T)$ if and only if $\R(T)\subset \mathcal{N}(ST-TS).$
\end{prop}
\begin{proof} Obvious.
\end{proof}
\begin{cor}\label{corn.1}  Let $S,T \in L(X).$ The following statements hold:\\
 (i)  If  $T$  is one-to-one, then   $TS \in \mbox{comm}(T)$  if and only if $S\in \mbox{comm}(T).$\\
(ii)  If  $T$  is onto, then   $ST \in \mbox{comm}(T)$  if and only if $S\in \mbox{comm}(T).$\\
(iii) Moreover,  if  $T$ and  $S$ are self-adjoint   Hilbert space operators, then $TS \in \mbox{comm}(T)$ if and only if $ST \in \mbox{comm}(T).$
\end{cor}
\begin{ex} Note that if  an operator $T$  is not onto and $ST \in \mbox{comm}(T)$,   then we cannot guarantee that $S$ commutes with $T$ even if $T$ is one-to-one. For this, consider  the  unilateral  right shift    $R$    and the nilpotent operator $N$  defined  on the Hilbert space $\ell^{2}$ by $Rx=(0, x_{1}, x_{2},\dots),$     $Nx=(0, -x_{1},0,\dots),$ where $x=(x_{n})_{n\geq1}\in\ell^{2}.$   $R$  is   one-to-one and not  onto, and  $NR \in \mbox{comm}(R).$ But $NR\neq RN.$   This   entails   also  from Proposition \ref{propn.1}     that the condition of  the injectivity of  $T$ assumed in the first assertion of Corollary \ref{corn.1}   is crucial.
\end{ex}
\par  Recall that the degree of stable iteration of an operator $T$ is defined  by  $\mbox{dis}(T)=\mbox{inf}\,\Delta(T),$     where    $$\Delta(T)=\{m\in\N \,:\,  \alpha(T_{[m]})= \alpha(T_{[r]}),\,\forall r \in \N  \,  \, r\geq m \}.$$   $T$ is said to be  semi-regular if $\R(T)$ is closed and $\mbox{dis}(T)=0,$ and      $T$ is said to be  essentially semi-regular if $\R(T)$ is closed  and there exists   a finite-dimensional subspace $F$ such that  $ \mathcal{N}(T)\subset \R(T^{\infty})+F.$ For more details about these definitions, one can see \cite{mbekhta,mullerreg}.
\begin{prop}\label{propn.2} Let   $S,T \in L(X)$ such that $S \in \mbox{comm}_{r}(T).$ The following assertions hold:\\
(i) If $\mbox{dis}(TS)=0,$ then $\mbox{dis}(S)=0$ and $\mbox{dis}(T)\leq 1.$\\
(ii) If  $TS$ is semi-regular, then $S$ is semi-regular.\\
(iii) If  $TS$ is essentially semi-regular, then $S$ is essentially  semi-regular.
\end{prop}
\begin{proof} (i) As   $S \in \mbox{comm}_{r}(T)$    we then get from Lemma \ref{lem1}  that   $(TS)^{n}=T^{n}S^{n}=ST^2S(TS)^{n-2}$ for all integer $n \geq 2.$  Moreover,    $\mbox{dis}(TS)=0$ implies that  $\mbox{dis}((TS)^{m})=0$ for every  $m\geq 1.$   Hence $\mathcal{N}(S^{m})\subset  \mathcal{N}((TS)^{m}) \subset  \bigcap_{n} \R((TS)^{nm})\subset \R(S)$ for all $m\geq 1.$ Hence    $\mbox{dis}(S)=0.$ Let $n\geq 1.$ As $\mbox{dis}(TS)=0,$ then  $\mathcal{N}(T^{n+1})\subset   \mathcal{N}(ST^{n+1})=  \mathcal{N}(TST^{n})\subset \mathcal{N}(T^{n})+ \R(TS^{n+1}).$ Therefore, $\mathcal{N}(T^{n+1})\subset \mathcal{N}(T^{n})+ \R(T)$ and then $\mbox{dis}(T)\leq 1.$ The points (ii) and (iii) are consequences  of  \cite[Lemme  4.15]{mbekhta},  \cite[Corollary 3.4, Theorem 3.5]{mullerreg} and    the fact that  $S(T^{2}S)=(TS)^{2}=(T^{2}S)S.$
\end{proof}
The next corollary  extends    \cite[Theorem 3.5]{mullerreg} and \cite[Proposition 3.7, Lemme 4.15]{mbekhta}.
\begin{cor} If  $T, S \in L(X)$ such that $S \in \mbox{comm}_{w}(T)$  and $TS$ is semi-regular (resp., essentially semi-regular),  then $ST,$ $T$ and $S$ are also semi-regular (resp., essentially semi-regular).
\end{cor}
\begin{proof}  As $S \in \mbox{comm}_{w}(T)$ then $(TS)^{2}=(ST)^{2}.$  Hence  $TS$ is semi-regular (resp., essentially semi-regular)  if and only if $(TS)^{2}$ is semi-regular (resp., essentially semi-regular) if and only if $(ST)^{2}$ is semi-regular (resp., essentially semi-regular)  if and only if $ST$ is semi-regular (resp., essentially semi-regular).    The rest of the proof follows directly from  Proposition \ref{propn.2}.
\end{proof}

\begin{prop}\label{propn.3} Let   $T \in L(X)$   and   $N \in \mbox{Nil}(L(X)).$ The following holds:\\
(i) If  $N\in\mbox{comm}_{r}(T),$ then    $T$ is onto if and only if $T +N$ is onto.\\
(ii) If  $N\in\mbox{comm}_{l}(T),$ then    $T$ is bounded below if and only if $T +N$ is bounded below.
\end{prop}
\begin{proof} Under  conditions assummed,  Corollary \ref{corn.1}  implies that  $N\in\mbox{comm}(T).$ And the results are already done. 
\end{proof}
\begin{lem}\label{lemMinv} Let $T,S\in L(X)$  such that $ST \in \mbox{comm}(T)$ and let $\lambda\neq 0.$ Then $M:=\mathcal{N}(T-\lambda I)$ is $S$-invariant and $S_{M}\in\mbox{comm}(T_{M}).$ If in addition $TS \in \mbox{comm}(S),$ then $B:=\mathcal{N}(T+S-\lambda I)$ is $S$-invariant and $S_{B}\in\mbox{comm}(T_{B}).$
\end{lem}
\begin{proof} Let $x\in M,$ then $(T-\lambda I)S(\lambda x)=(T-\lambda I)ST(x)=ST(T-\lambda I)(x)=0.$ Thus $M$ is $S$-invariant. On the other hand, as $T_{M}$ is invertible and $ST \in \mbox{comm}(T),$ then $S_{M}\in\mbox{comm}(T_{M}).$  If in addition $TS \in \mbox{comm}(S),$ then $S(T+S) \in \mbox{comm}(T+S)$ and  thus $B$ is $S$-invariant and $T$-invariant. Therefore $S_{B}\in\mbox{comm}(T_{B}).$
\end{proof}
\begin{thm}\label{thmn.1} Let    $T \in L(X)$   and   $N \in \mbox{Nil}(L(X))$     such that   $T\in \mbox{comm}(NT).$     Then    
  $$\sigma_{p}(T)\setminus\{0\}\subset \sigma_{p}(T+N)\setminus\{0\}.$$
\end{thm}
\begin{proof}  Let $\lambda \neq 0$ and let $x\in \mathcal{N}(T-\lambda I).$  We  show by induction that  $((T+N)-\lambda I)^{n}(x)=N^{n}(x)$ for any  $n\in \N.$  Indeed, $((T+N)-\lambda I)(x)=N(x).$ Assume   that $((T+N)-\lambda I)^{n}(x)=N^{n}(x)$ for some positive  integer $n.$ Then $((T+N)-\lambda I)^{n+1}(x)=((T+N)-\lambda I)(N^{n}(x))=TN^{n}(x)+N^{n+1}(x)-\lambda N^{n}(x).$ Furthermore,   Lemma  \ref{lemMinv}  implies that   $TN^{m}(x)=\lambda N^{m}x$ for all $m\in \N.$  Hence  $((T+N)-\lambda I)^{n+1}(x)=N^{n+1}(x).$  Let $p\geq1$  such that $N^{p}=0,$ then $((T+N)-\lambda I)^{p}(x)=0$ and thus  $x\in \mathcal{N}(((T+N)-\lambda I)^{p}).$ This yields   $\mathcal{N}(T-\lambda I) \subset  \mathcal{N}((T+N)-\lambda I)^{p}).$  Hence $\sigma_{p}(T)\setminus\{0\}\subset \sigma_{p}(T+N)\setminus\{0\}.$ 
\end{proof}
From   the proof of  Theorem \ref{thmn.1}, we obtain the next proposition.
\begin{prop}\label{propkernel} Let $T, N \in L(X)$     such that   $T\in \mbox{comm}(NT)$ and   $N^{p}=0$ for some  strictly positive integer $p.$  Then for every  $\lambda\neq 0,$ we have  $\mathcal{N}(T-\lambda I) \subset  \mathcal{N}((T+N)-\lambda I)^{p}).$ If in addition  $N \in \mbox{comm}(TN),$ then $\mathcal{N}((T+N)-\lambda I) \subset  \mathcal{N}((T-\lambda I)^{p}).$
\end{prop}
Note that   in the case of   $NT\in \mbox{comm}(T)$ and $N^{2}=0,$ the  following proposition shows  (without the condition   $TN\in \mbox{comm}(N)$) that  $\mathcal{N}((T+N)-\lambda I) \subset  \mathcal{N}((T-\lambda I)^{2}),$  which implies in turn that   the  inclusion proved in  Theorem \ref{thmn.1} becomes equality.
\begin{prop}\label{propn2zero} Let    $T \in L(X)$   and   $N \in \mbox{Nil}(L(X))$     such that   $NT\in \mbox{comm}(T)$ and $N^{2}=0.$     Then $\sigma_{p}(T)\setminus\{0\}= \sigma_{p}(T+N)\setminus\{0\}$ and $\sigma_{p}^0(T)\setminus\{0\}=\sigma_{p}^0(T+N)\setminus\{0\}.$
\end{prop}
\begin{proof} Let $\lambda \neq 0$ and let us to show  that  $\mathcal{N}(T+N-\lambda I) \subset  \mathcal{N}((T-\lambda I)^{2}).$ Let $x\in \mathcal{N}(T+N-\lambda I),$ then $(T+N)x=\lambda x.$  So $\lambda^{2}(T-\lambda I)^{2}x=\lambda^{2}(T-\lambda I)(-N)x=\lambda^{2}(-TNx+N(\lambda x)) =\lambda^{2}(-TNx+N(T+N)x)= \lambda^{2}(N^{2}x+(NT-TN)x).$ Moreover, we have   $\lambda NTx=NT^{2}x+NTNx=TNTx+NTNx=TN(\lambda I-N)x+NTNx=\lambda TNx-TN^{2}x+NTNx$ and  $\lambda NTNx=NTNTx+NTN^{2}x=N^{2}T^{2}x+NTN^{2}x.$ Hence  $\lambda^{2}(T-\lambda I)^{2}x=\lambda^{2}N^{2}x+N^{2}T^{2}x+NTN^{2}x-\lambda TN^{2}x=0$ and then    $x\in  \mathcal{N}((T-\lambda I)^{2}).$ On the other hand, from proposition \ref{propkernel}, we have $\mathcal{N}(T+N-\lambda I) \subset  \mathcal{N}((T-\lambda I)^{2}).$  Hence $\sigma_{p}(T)\setminus\{0\}= \sigma_{p}(T+N)\setminus\{0\}$ and $\sigma_{p}^0(T)\setminus\{0\}=\sigma_{p}^0(T+N)\setminus\{0\}.$
\end{proof}
\begin{cor}\label{corn.2}  Let   $T \in L(X)$   and   $N \in \mbox{Nil}(L(X)).$  The following  assertions hold:\\
(i) If  $T \in \mbox{comm}(NT)\cap\mbox{comm}(TN),$  then $\sigma_{p}(T)\setminus\{0\}\subset\sigma_{p}(T+N)\setminus\{0\}$ and    $\sigma_{p}(T^{*})\setminus\{0\}\subset\sigma_{p}(T^{*}+N^{*})\setminus\{0\}.$ \\
(ii) If  $N\in\mbox{comm}_{r}(T),$ then  $\sigma_{p}(T)\setminus\{0\}=\sigma_{p}(T+N)\setminus\{0\} \text{ and } \sigma_{p}^0(T)\setminus\{0\}=\sigma_{p}^0(T+N)\setminus\{0\}.$
\end{cor}

\begin{proof} (i) is obvious and  (ii) is a consequence of Proposition \ref{propkernel}.
\end{proof}
The condition assumed in assertions (ii) of the previous corollary cannot guarantee  that    $\sigma_{p}(T)=\sigma_{p}(T+N)$ or  $\sigma_{p}^{0}(T)= \sigma_{p}^{0}(T+N),$ as the following examples shows.
\begin{ex}\label{exnilpotent} Let $T, N \in L(\ell^{2})$ be the operators defined by $T(x)=(0,\frac{x_{1}}{2},\frac{x_{2}}{3},\dots),$  $N(x)=(0,\frac{-x_{1}}{2}, 0,\dots)$ for every  $x=(x_{n})_{n\geq 1} \in \ell^{2}.$ Clearly, $T$ and $T+N$ are  quasi-nilpotent and  compact,  $N$ is nilpotent. Moreover, $TNT=NT^{2}=NT=N^{2}T=NTN=TN^{2}\neq T^{2}N$  and $\sigma_{p}(T)=\sigma_{p}^{0}(T)=\emptyset\neq \{0\}=\sigma_{p}^{0}(T+N)=\sigma_{p}(T+N).$   If we take the nilpotent operator $Q \in L(\ell^{2})$  defined by  $Q(x)=(0,\frac{-x_{1}}{2}, 0,\frac{-x_{3}}{4},0, \frac{-x_{5}}{6},\dots),$ then $(T+Q)^{2}=0,$   $TQT=QT^{2}=QT=Q^{2}T=QTQ=TQ^{2}\neq T^{2}Q$ and     $\sigma_{p}(T)=\sigma_{p}^{0}(T)=\sigma_{p}^{0}(T+Q)=\emptyset\neq \{0\}=\sigma_{p}(T+Q).$
\end{ex}
\begin{rema}  Let $T \in L(X).$ It is well known that  $\sigma_{p}(T)=\sigma_{p}(T+N)$  for every operator   $N \in \mbox{comm}(T)\cap\mbox{Nil}(L(X)).$ This result cannot be extended  for  operator  $N \in \displaystyle\left[\mbox{comm}(T^{2})\cap\mbox{Nil}(L(X))\right]\cup \left[\mbox{Nil}(L(X))\cap N^{-1}(\mbox{comm}(T))\right],$   as the following  shows.    The nilpotent operators $T$ and $N$ defined in the point (iii) of the     Example \ref{s.ex} satisfy $\emptyset=\sigma_{p}(T)\setminus\{0\}=\sigma_{p}(N)\setminus\{0\}\neq \{-1,1\}=\sigma_{p}(T+N)\setminus\{0\},$ although  $TN^{2}=N^{2}T=NT^{2}=T^{2}N=0.$ Note also that $\sigma_{p}(S)=\sigma_{a}(S)=\sigma_{s}(S)=\sigma(S)$ for all $S\in\{T, N, T+N\}.$
\end{rema}

To give further information about the approximate point spectrum  of sums of operators we need to introduce the {\it Berberian-Quisley extension} \cite{berberian,rickart}. Consider     $\ell^{\infty}(X)$   the Banach space   of all   bounded sequences $x = (x_{n})$ of  $X$ by imposing term-by-term linear combination and the supremum norm $\|x\| = \mbox{sup}\|x_{n}\|.$  Then the quotient space  $X_{0}=\ell^{\infty}(X)/c_{0}(X)$ is a Banach space, where  $c_{0}(X)=\{ (x_{n})\subset  X :  \mbox{lim}\, \|x_{n}\|=0\}.$ Any operator $T \in L(X)$  generates an operator $T^{0} \in L(X_{0})$ defined by  $T^{0}(x+c_{0}(X))=(Tx_{n})_{n}+c_{0}(X)$ for every $x = (x_{n}) \in \ell^{\infty}(X).$  The  mapping $T$ $\longrightarrow$ $T^{0}$  of $L(X)$ into $L(X_{0})$  is an isometric isomorphism and $\sigma_{a}(T) =\sigma_{a}(T^{0}) = \sigma_{p}(T^{0}).$
\begin{prop}\label{propn.aaa}
Let  $T \in L(X)$  and let  $N\in \mbox{Nil}(L(X)).$ The following assertions hold:\\
(i)  If $T \in \mbox{comm}(NT),$ then  $\sigma_{a}(T)\setminus\{0\}\subset \sigma_{a}(T+N)\setminus\{0\},$ and  if in addition   $N^{2}=0,$  then   $\sigma_{a}(T)\setminus\{0\}= \sigma_{a}(T+N)\setminus\{0\}.$    While  if   $T \in \mbox{comm}(TN),$ then $\sigma_{s}(T)\setminus\{0\}\subset \sigma_{s}(T+N)\setminus\{0\},$ and   if in addition   $N^{2}=0,$  then   $\sigma_{s}(T)\setminus\{0\}= \sigma_{s}(T+N)\setminus\{0\}.$  \\
(ii)  If $T \in \mbox{comm}(NT)\cap\mbox{comm}(TN),$ then  $\sigma_{*}(T)\setminus\{0\}\subset \sigma_{*}(T+N)\setminus\{0\},$ and if   in addition   $N^{2}=0,$  then   $\sigma_{*}(T)\setminus\{0\}=\sigma_{*}(T+N)\setminus\{0\},$   where $\sigma_{*}\in\{\sigma_{a},\sigma_{s},\sigma\}.$\\
(iii)  If   $N\in\mbox{comm}_{r}(T),$ then $\sigma_{a}(T)\setminus\{0\}=\sigma_{a}(T+N)\setminus\{0\}.$\\
(iv)  If   $N\in\mbox{comm}_{l}(T),$ then $\sigma_{s}(T)\setminus\{0\}=\sigma_{s}(T+N)\setminus\{0\}.$\\
(v) If $N \in \mbox{comm}_{w}(T),$  then   $\sigma_{*}(T)=\sigma_{*}(T+N),$  where $\sigma_{*}\in\{\sigma_{p},\sigma_{p}^{0},\sigma_{a},\sigma_{s},\sigma\}.$
\end{prop}

\begin{proof}
(i) Since $T \in \mbox{comm}(NT)$  then $T^{0}N^{0}T^{0}=(TNT)^{0}=(NT^2)^{0}=N^{0}(T^{0})^{2}.$  So  $T^{0}\in \mbox{comm}(N^{0}T^{0}).$  Moreover,  $\|(N^{0})^{p}\|=\|(N^{p})^{0}\|=0$ and then $N^{0}$ is   nilpotent. From Theorem  \ref{thmn.1}, $\sigma_{a}(T)\setminus\{0\}=\sigma_{p}(T^{0})\setminus\{0\}\subset\sigma_{p}(T^{0}+N^{0})\setminus\{0\}=\sigma_{p}((T+N)^{0})\setminus\{0\}=\sigma_{a}(T+N)\setminus\{0\}.$ If in  addition   $N^{2}=0,$  then  we deduce from Corollary \ref{propn2zero}   that  $\sigma_{a}(T)\setminus\{0\}= \sigma_{a}(T+N)\setminus\{0\}.$   While if  $T \in \mbox{comm}(TN),$ then $T^{*} \in \mbox{comm}(N^{*}T^{*})$ and thus  $\sigma_{s}(T)\setminus\{0\}=\sigma_{a}(T^{*})\setminus\{0\}\subset\sigma_{a}(T^{*}+N^{*})\setminus\{0\}=\sigma_{a}((T+N)^{*})\setminus\{0\}=\sigma_{s}(T+N)\setminus\{0\},$ and if in addition $N^{2}=0,$  then   $\sigma_{s}(T)\setminus\{0\}= \sigma_{s}(T+N)\setminus\{0\}.$   The point (ii) follows directly from the  first.\\
(iii)   As    $N\in\mbox{comm}_{r}(T)$ then  $T \in \mbox{comm}(NT)$ and  $(T+N) \in \mbox{comm}(-N(T+N)).$ So the  first point gives the desired result.  The proof of  (iv) goes similarly with (i), while the proof  of   (v)   is a consequence of  (iii), (iv) and  Proposition \ref{propn.3}.
\end{proof}
\begin{cor} Let $\A$ be an  arbitrary unital complex Banach algebra and let $x\in \A$ and $a\in \mbox{Nil}(\A).$ The following statements hold:\\
(i) If $x \in\mbox{comm}(ax)\cap\mbox{comm}(xa),$  then $\sigma(x)\setminus\{0\}\subset\sigma(x+a)\setminus\{0\}.$ If in addition $a^{2}=0,$ then  $\sigma(x)\setminus\{0\}=\sigma(x+a)\setminus\{0\}.$\\
(ii) If $x \in\mbox{comm}_{w}(a),$  then $\sigma(x)=\sigma(x+a).$
\end{cor}
\begin{proof} Consider  the operators   $L_{a}(y)=ay$ and $L_{x}(y)=xy.$  We have  $L_{x}\in L(\A),$  $L_{a}\in \mbox{Nil}(L(\A))$ and $\sigma(x)=\sigma(L_{x}).$ Remark   that if $a^{2}=0,$ then $L_{a}^{2}=L_{a^{2}}=0.$ And if $x \in\mbox{comm}(ax)\cap\mbox{comm}(xa)$ (resp.,  $x \in\mbox{comm}_{w}(a)$), then $L_{x} \in\mbox{comm}(L_{a}L_{x})\cap\mbox{comm}(L_{a}L_{x})$ (resp.,  $L_{x} \in\mbox{comm}_{w}(L_{a})$). By applying   Proposition \ref{propn.aaa}, we get the desired results.
\end{proof}
\medskip

In this paragraph we present the  {\it Construction of Sadovskii/Buoni, Harte, Wickstead} \cite{buoni,mullerbook,sadovskii}, which will play an important role in the next.   The space   $m(X)$   consisting of the relatively compact sequences  of $X$   is a closed  subspace of $\ell^{\infty}(X).$ Consider   $P(T) \in L(\mathcal{P}(X))$  the operator defined by $P(T)(x+m(X))=(Tx_{n})_{n}+m(X),$ where  $x=(x_{n}) \in \ell^{\infty}(X)$ and $\mathcal{P}(X)=\ell^{\infty}(X)/m(X).$   The  mapping $T$ $\longrightarrow$ $P(T)$  of $L(X)$ into $L(\mathcal{P}(X))$  is a unital homomorphism  with kernel  $K(X)$ 	  and  induces a norm decreasing monomorphism from $L(X)/K(X)$ to $L(X).$ Moreover, $\|P(T)\|\leq \|T\|,$ $\sigma_{uf}(T)=\sigma_{a}(P(T)),$ $\sigma_{lf}(T)=\sigma_{s}(P(T))$ and $\sigma_{e}(T)=\sigma(P(T)).$
\begin{prop}\label{propn.aaaa}
Let  $T,K \in L(X)$   such that   $K$ is a power compact operator. The following assertions hold:\\
(i)  If $T \in \mbox{comm}(KT),$ then  $\sigma_{*}(T)\setminus\{0\}\subset \sigma_{*}(T+K)\setminus\{0\},$ where $\sigma_{*}\in\{\sigma_{uf},\sigma_{uw}\},$   and if in addition   $K^{2}$ is compact,  then   $\sigma_{**}(T)\setminus\{0\}= \sigma_{**}(T+K)\setminus\{0\},$ where $\sigma_{**}\in\{\sigma_{uf},\sigma_{uw},\sigma_{ub}\}.$ While  if   $T \in \mbox{comm}(TK),$ then $\sigma_{+}(T)\setminus\{0\}\subset \sigma_{+}(T+K)\setminus\{0\},$ where $\sigma_{+}\in\{\sigma_{lf},\sigma_{lw}\},$    and   if in addition   $K^{2}$ is compact,  then $\sigma_{++}(T)\setminus\{0\}= \sigma_{++}(T+K)\setminus\{0\},$ where $\sigma_{++}\in\{\sigma_{lf},\sigma_{lw},\sigma_{lb}\}.$  \\
(ii)  If $T \in \mbox{comm}(KT)\cap\mbox{comm}(TK),$ then   $\sigma_{*}(T)\setminus\{0\}\subset \sigma_{*}(T+K)\setminus\{0\},$ where $\sigma_{*}\in\{\sigma_{uf},\sigma_{uw},\sigma_{lf},\sigma_{lw},\\\sigma_{e},\sigma_{w}\},$   and if in addition   $K^{2}$ is compact,  then   $\sigma_{**}(T)\setminus\{0\}=\sigma_{**}(T+K)\setminus\{0\},$ where $\sigma_{**}\in\{\sigma_{uf},\sigma_{uw},\sigma_{ub},\sigma_{lf},\sigma_{lw},\sigma_{lb},\sigma_{e},\sigma_{w},\sigma_{b}\}.$\\
(iii)  If      $K\in\mbox{comm}_{r}(T),$ then $\sigma_{*}(T)\setminus\{0\}=\sigma_{*}(T+K)\setminus\{0\},$ where $\sigma_{*}\in\{\sigma_{uf},\sigma_{uw},\sigma_{ub}\}.$ \\
(iv)  If $K\in\mbox{comm}_{l}(T),$ then $\sigma_{*}(T)\setminus\{0\}=\sigma_{*}(T+K)\setminus\{0\},$ where $\sigma_{*}\in\{\sigma_{lf},\sigma_{lw},\sigma_{lb}\}.$\\
(v) If $K\in\mbox{comm}_{w}(T),$ then $\sigma_{*}(T)=\sigma_{*}(T+K),$  where $\sigma_{*}\in\{\sigma_{e},\sigma_{w},\sigma_{b},\sigma_{gd},\sigma_{g_{z}d}\}.$

\end{prop}
\begin{proof}
(i)   $T \in \mbox{comm}(KT)$    implies that $P(T)\in\mbox{comm}(P(K)P(T)).$  Let    $p\geq 1$  such that $K^{p}$ is compact, we have      $P(0)=P(K^{p})=P(K)^{p},$ and  so   $P(K)$ is   nilpotent.  From  \cite[Theorem 2]{buoni} and Proposition \ref{propn.aaa},  $\sigma_{uf}(T)\setminus\{0\}=\sigma_{a}(P(T))\setminus\{0\}\subset\sigma_{a}(P(T)+P(K))\setminus\{0\}=\sigma_{a}(P(T+K))\setminus\{0\}=\sigma_{uf}(T+K)\setminus\{0\}.$  Let $\lambda \notin \sigma_{uw}(T+K)\setminus\{0\},$ then $\lambda \notin  \sigma_{uf}(T)\setminus\{0\}.$   By using  the same argument as Oberai in \cite[Lemma 2]{oberai},  we get that $\mbox{ind}(T+K-\lambda I)=\mbox{ind}(T-\lambda  I)$ and  thus $\lambda \notin \sigma_{uw}(T)\setminus\{0\}.$  Therefore $\sigma_{uw}(T)\setminus\{0\}\subset\sigma_{uw}(T+K)\setminus\{0\}.$  Since   $\sigma_{ub}(T)=\sigma_{uw}(T)\cup\mbox{iso}\,\sigma_{a}(T)$ then if in addition   $K^{2}$ is compact,     $\sigma_{ub}(T)\setminus\{0\}= \sigma_{ub}(T+K)\setminus\{0\}.$ The rest of the proof is   clear and is left to the reader.
\end{proof}

We recall that   $T\in L(X)$ is said to have the SVEP  at $\lambda\in\mathbb{C}$ if  for every open neighborhood $U_\lambda$ of $\lambda,$ the  function $f\equiv 0$ is the only  analytic solution of the equation $(T-\mu I)f(\mu)=0\quad\forall\mu\in U_\lambda.$
\begin{lem}\label{lemf.1}  Let  $T \in L(X).$ Then   $T$ is quasi-nilpotent if and only if  $\sigma_{*}(T)=\{0\},$ where  $\sigma_{*}\in \{\sigma_{a},\sigma_{s}\}.$
\end{lem}
\begin{proof}
Since  $\partial\sigma(T)\subset\sigma_{*},$ then the proof follows  from \cite[Theorem 2.97, Theorem 2.98]{Aiena1} and the fact that   $T$ and $T^{*}$  have the SVEP on the boundary $\partial\sigma(T).$
\end{proof}
Recall that $T\in L(X)$ is   Riesz   if $T-\lambda I$ is  Browder    for all non-zero complex  $\lambda,$  which is equivalent to say that $\pi(T):=T+K(X)$ is quasi-nilpotent in the Calkin algebra $L(X)/K(X),$ where $K(X)$ is the  ideal  of all compact operators.
\begin{prop}\label{prop2f.1}   Let $T,R\in L(X)$ such that $R$ is Riesz. The following statements hold:\\
(i)   If $T \in \mbox{comm}(RT)\cap\mbox{comm}(TR)$ and        $T$ is Fredholm,  then      $\sigma_{e}(T)=\sigma_{e}(T+R)$ and $\sigma_{w}(T)=\sigma_{w}(T+R).$\\
(ii)  If   $T \in \mbox{comm}(TR)$  and      $T$ is upper semi-Fredholm, then   $\sigma_{uf}(T)=\sigma_{uf}(T+R)$ and $\sigma_{uw}(T)=\sigma_{uw}(T+R).$\\
(iii)  If  $T \in \mbox{comm}(RT)$  and      $T$ is lower semi-Fredholm, then   $\sigma_{lf}(T)=\sigma_{lf}(T+R)$ and $\sigma_{lw}(T)=\sigma_{lw}(T+R).$
\end{prop}
\begin{proof} (i) Assume that  $T$ is Fredholm.  By   Atkinson theorem we get  that  $\pi(T)$ is invertible in the Calkin algebra. As  $TRT\in \{T^{2}R,RT^{2}\}$  then $\pi(T)\pi(R)=\pi(R)\pi(T)$ and   thus $TR-RT\in K(X).$ Since $R$ is Riesz then $\pi(R)$ is quasi-nilpotent, and hence  $\sigma_{e}(T)=\sigma(\pi(T))=\sigma(\pi(T+R))=\sigma_{e}(T+R).$ And thus $\sigma_{w}(T)=\sigma_{w}(T+R).$\\
(ii)  If  $T$ is upper semi-Fredholm, then from \cite{buoni} the operator $P(T)\in \ell^{\infty}(X)/m(X)$ defined above is bounded below.    As $TRT=T^{2}R,$ from Corollary \ref{corn.1},   $P(T)P(R)=P(R)P(T).$ Thus $P(TR-RT)=P(0),$ so  that $TR-RT\in K(X).$  On the other hand, since   $R$ is Riesz then  $\sigma_{a}(P(R))=\sigma_{uf}(R)=\{0\}$ and this implies from Lemma \ref{lemf.1}  that $P(R)$  is quasi-nilpotent. Hence $\sigma_{uf}(T)=\sigma_{a}(P(T))=\sigma_{a}(P(T+R))=\sigma_{uf}(T+R).$    Let $\lambda \notin \sigma_{uw}(T),$ then $\lambda \notin \sigma_{uf}(T)=\sigma_{uf}(T+\mu R)$ for all $\mu \in \C.$   From \cite[Theorem V.I.8]{Goldberg}, we deduce that $\alpha(T-\lambda I)=  \alpha((T+\mu R)-\lambda I)$ and $\beta(T-\lambda I)= \beta((T+\mu R)-\lambda I)$ for all $\mu \in \C.$ Hence $\mbox{ind}(T-\lambda I)=\mbox{ind}((T+ R)-\lambda I)$ for all $\mu \in \C.$ Consequently,  $\lambda \notin \sigma_{uw}(T+R).$  The point  (iii) goes similarly.\\
\end{proof}

\goodbreak
{\small \noindent Zakariae Aznay,\\  Laboratory (L.A.N.O), Department of Mathematics,\\Faculty of Science, Mohammed I University,\\  Oujda 60000 Morocco.\\
aznay.zakariae@ump.ac.ma\\

\noindent Abdelmalek Ouahab,\newline Laboratory (L.A.N.O), Department of
	Mathematics,\newline Faculty of Science, Mohammed I University,\\
	\noindent Oujda 60000 Morocco.\\
	\noindent ouahab05@yahoo.fr\\

 \noindent Hassan  Zariouh,\newline Department of
Mathematics (CRMEFO),\newline
 \noindent and laboratory (L.A.N.O), Faculty of Science,\newline
  Mohammed I University, Oujda 60000 Morocco.\\
 \noindent h.zariouh@yahoo.fr

\end{document}